\documentclass[12pt,reqno,intlim]{amsart}

\usepackage{amssymb,amsmath}

\usepackage[arrow,matrix,curve]{xy}

\newdir{ >}{{}*!/-5pt/\dir{>}}

\newcommand{\nc}{\newcommand}

\nc{\bC}{\bold{C}} \nc{\bN}{\Bbb{N}} \nc{\cF}{\mathcal{F}}
\nc{\cE}{\mathcal{E}} \nc{\cR}{\mathcal{R}} \nc{\cM}{\mathcal{M}}
\nc{\al}{\alpha} \nc{\bt}{\beta} \nc{\gm}{\gamma} \nc{\dl}{\delta}
\nc{\om}{\omega} \nc{\sg}{\sigma} \nc{\Sg}{\Sigma} \nc{\vf}{\varphi}
\nc{\ve}{\varepsilon} \nc{\os}{\overset} \nc{\ol}{\overline}
\nc{\ul}{\underline} \nc{\us}{\underset} \nc{\sbs}{\subset}
\nc{\bsl}{\backslash} \nc{\Ra}{\Rightarrow}
\nc{\lra}{\longrightarrow} \nc{\all}{\allowdisplaybreaks}

\nc{\Codes}{\operatorname{{\bold{Codes}}}}
\nc{\RegMono}{\operatorname{\mathcal{R}{\rm{eg}\mathcal{M}{\rm{ono}\!}}}}
\nc{\RegEpi}{\operatorname{\mathcal{R}{\rm{eg}\mathcal{E}{\rm{pi}\!}}}}
\nc{\Mn}{\operatorname{\mathcal{M}{\rm{ono}\!}}}
\nc{\Ep}{\operatorname{\mathcal{E}{\rm{pi}\!}}}
\nc{\Rg}{\operatorname{\mathcal{R}{\rm{eg}\!}}}
\nc{\Ob}{\operatorname{Ob\!}}

\numberwithin{equation}{section}

\newtheorem{theo}{\ \ \ Theorem}[section]
\newtheorem{lem}[theo]{\ \ \ Lemma}
\newtheorem{prop}[theo]{\ \ \ Proposition}
\newtheorem{cor}[theo]{\ \ \ Corollary}
\newtheorem{definition}[theo]{\ \ \ Definition}

\theoremstyle{definition}
\newtheorem{exmp}[theo]{\ \ \ Example}

\theoremstyle{remark}
\newtheorem{rem}[theo]{\ \ \ Remark}

\begin{document}

\title[]
{Associative Protomodular Algebras}

\author{Dali Zangurashvili}

\maketitle

\begin{abstract}
 The
notion of associativity (which differs from the straightforward
generalization of the usual associativity given by the move of
parentheses in the relevant expression) for operations of high arity
is introduced. It is proved that the algebraic theory of a variety
of universal algebras contains a group operation if and only if it
contains a semi-abelian operation which is associative in the sense
introduced.
\bigskip


\end{abstract}

\section{Introduction}


 In \cite{JP}, among many other things P. T. Johnstone and C. M. Pedicchio noted that
 an algebraic theory (with a constant) contains
 a group operation if and only if it contains a
 Mal'cev operation which is associative in the sense of \cite{JP}. Note also that there are two important classes of Mal'cev varieties
-- the classes of protomodular and semi-abelian varieties --
  which, similarly to the case of Mal'cev varieties,
have purely syntactical characterizations due to D. Bourn and G.
Janelidze \cite{BJ}. Recall that the notions of protomodular and
semi-abelian varieties were derived from the corresponding
categorical notions introduced by D. Bourn \cite{B} and G.
Janelidze, L. Marki, and W. Tholen \cite{JMT}, respectively, as
abstract settings in which many properties of groups remain valid.
The Bourn-Janelidze characterizations require the existence of one
operation of arbitrarily high arity (called a
protomodular/semi-abelian operation), together with some binary
operations
 and constants that satisfy certain identities
 (which, in fact, were originally considered by A.Ursini \cite{U}, but for different
 purposes). In view of the above-said a natural question arises whether there is an
analog
 of the Johnstone-Pedicchio statement which would impose an "associativity-like" condition on
 a protomodular/semi-abelian operation.

 Since the Bourn-Janelidze characterizations enable us to construct explicitly the Mal'cev term
in a protomodular/semi-abelian
 variety, one can attempt to find such an analog just by plugging the explicit form of the Mal'cev term in
the associativity identity mentioned in the Johnstone-Pedicchio
statement. However, the condition obtained in this way does not
provide the basis to be interpreted as a kind of higher
associativity, since, for instance,
 it contains not only the protomodular/semi-abelian operation, but
 also the binary operations and constants
 from the Bourn - Janelidze's result.

 In the present paper, we give the
 analog of the Johnstone-Pedicchio statement whose formulation requires only a kind of higher
 associativity condition on a semi-abelian term, and does not involve any other operations. To this end, we first introduce the notion of 2-associativity for
operations of high arity which is a generalization of the usual
associativity condition (but is different from its straightforward
generalization given by the move of parentheses in the relevant
expression; the term "1-associativity" is left for the latter
generaization). The point here is that the associativity for a
binary operation $\times$ on a set $A$ can be formulated as the
condition that the mapping $(-\times a)$ from $A$ to the algebra of
mappings $Map(A,A)$ (with the composition operation) preserves the
operation $\times$, for any $a\in A$. Generalizing this condition to
the case of operations $\theta$ of high arity we get the notion of
2-associativity (see Section 3). It is equivalent to the identity

$$\theta(a_{1},a_{2},...,a_{n},\theta(b_{1},b_{2},...,b_{n},b))=$$
$$=\theta(\theta(a_{1},a_{2},...,a_{n},b_{1}),\theta(a_{1},a_{2},...,a_{n},b_{2}),...,\theta(a_{1},a_{2},...,a_{n},b_{n}),b).$$\vskip+2mm

\noindent An example of such an operation is given by the
protomodular operation of the algebraic theory of the variety of
Boolean algebras: $\theta(a,b,c)=(a\vee c)\wedge b$. In fact, this
operation, like a similar one $\theta(a,b,c)=(a\vee b)\wedge c$, is
2-associative in any distributive lattice.

The main result of the present paper asserts that \emph{the
algebraic theory of a variety of universal algebras contains a group
operation if and only if it contains a 2-associative semi-abelian
term $\theta$. In that case the group operation is defined by}

$$a\cdot b=\theta(a,a,...,a,b),$$

\noindent \emph{its unit is $e$ and the inverse of an element $a$ is
given by}

\begin{equation}
a^{-1}=
\end{equation}
$$=\theta(\alpha_{1}(e,\theta(a,a,...,a)),\alpha_{2}(\theta(e,\theta(a,a,...,a)),...,\alpha_{n}(e,
\theta(a,a,...,a)),a),
$$
\noindent \emph{where $\alpha_{i}$ are the binary operations and $e$
is the constant from the Bourn-Janelidze characterization of
protomodular varieties.}

If we look at 2-associative algebras of the simplest semi-abelian
variety as high arity analogs of groups, then this result implies
that the "$n$-arity groups" are nothing else but the so-called
$n$-enriched groups (in the sense of Section 4 of this paper).

\section{Preliminaries}

For the definitions of a protomodular category and a semi-abelian
category we refer the reader to the papers \cite{B} by D. Bourn and
\cite{JMT} by G. Janelidze, L. Marki, and W. Tholen, respectively.

The characterizations of protomodular and semi-abelian varieties of
universal algebras were found by D. Bourn and G. Janelidze in
\cite{BJ}. Namely, they proved that
a variety of universal algebras $\mathbb{V}$ is protomodular if and
only if its algebraic theory contains, for some natural $n$,
constants $e_{1}, e_{2},...,e_{n}$, binary operations $\alpha_{1}$,
$\alpha_{2}$,..., $\alpha_{n}$ and an $(n+1)$-ary operation $\theta$
such that the following identities are satisfied:

\begin{equation}
\alpha_{i}(a,a)=e_{i};
\end{equation}

\begin{equation}
\theta(\alpha_{1}(a,b),\alpha_{2}(a,b),...,\alpha_{n}(a,b),b)=a.
\end{equation}
\vskip+3mm

\noindent $\mathbb{V}$ is semi-abelian if and only if its signature
contains a unique constant $e$, and (2.1) and (2.2) are satisfied
for $e_{1}=e_{2}=...=e_{n}=e$.

For simplicity, algebras from a protomodular (resp. semi-abelian)
variety are called protomodular (resp. semi-abelian). The operation
$\theta$ satisfying (2.2) for some $\alpha_{i}$ and $e_{i}$ which in
their turn satisfy (2.1) is called protomodular. A protomodular
operation is called semi-abelian if all $e_{i}$ are equal.

The motivating example of a semi-abelian variety is the variety of
groups; in that case we have
$$\theta(a,b)=ab,$$
$$\alpha_{1}(a,b)=a/b,$$
$$e_{1}=1.$$

\noindent Similarly, any variety whose algebraic theory contains a
group operation is protomodular. The varieties of Boolean algebras
and Heyting algebras are protomodular too. As is well-known, the
algebraic theory of the former variety has a group operation:
$$\theta(a,b)=(a\wedge \rceil b)\vee (b\wedge\rceil a).$$
\noindent Another protomodular operation of this algebraic theory
\cite{BC2} is given by :

$$\theta(x,y,z)=(x\vee z)\wedge y,$$
\noindent with

$$\alpha_{1}(x,y)=(x\wedge \rceil y),$$
$$\alpha_{2}(x,y)=(x\vee  \rceil y),$$
$$e_{1}=0,e_{2}=1,$$

\noindent For a protomodular operation of the variety of Heyting
algebras we refer the reader to \cite{J}; this operation in fact is
semi-abelian (nevertherless the variety of Heyting algebras is not
semi-abelian!). The same operation makes the variety of Heyting
semi-lattices semi-abelian \cite{J}. Other examples of semi-abelian
varieties are given by the varieties of loops, left/right
semiloops\footnote{For loops and left semiloops, the semi-abelian
operation and the relevant binary operation are the same as in the
case of groups.}, and locally Boolean distributive lattices
\cite{BC1}.

The identities (2.1) and (2.2)  imply \cite{BC2}:\vskip+2mm



\begin{equation}
\theta(e_{1},e_{2},...,e_{n},a)=a.
\end{equation}
\vskip+2mm


\section{Associativity Conditions for Operations of High Arity}

 Note that the associativity condition for a binary operation $\theta$ on a set  $A$ is equivalent to the condition that the mapping $F:A\rightarrow Map(A,A)$ defined as $F(b)=\theta_{b}=\theta(-,b)$
 is a homomorphism (when the set of all mappings $Map(A,A)$
is equipped with the composition operation). Taking this observation
into account, below we define the 2-associativity condition for
operations of any arities.

    Let $A$ be a set equipped with an $(n+1)$-ary operation $\theta$, and let $Map(A^{n},A)$ be the set of all mappings $A^{n}\rightarrow A$.
Let us introduce the $(n+1)$-ary operation $\theta$ on
$Map(A^{n},A)$ as follows: for any $f_{1},f_{2},...,f_{n},g\in
Map(A^{n},A)$, we define $\theta(f_{1},f_{2},...,f_{n},g)$ as the
composition of mappings $(f_{1},f_{2},...,f_{n}):A^{n}\rightarrow
A^{n}$ and $g:A^{n}\rightarrow A$.

           We have the mapping  $F:A\rightarrow Map(A^{n},A)$ defined as  $F(b)=\theta_{b}$, which sends an $n$-tuple $(a_{1},a_{2},...,a_{n})$ to $\theta(a_{1},a_{2},...,a_{n},b)$.

       \begin{lem} The following conditions are equivalent:

(a) The mapping  $F$ preserves $\theta$;

(b)  For any $a_{1},a_{2},...,a_{n}, b_{1},b_{2},...,b_{n},c\in A$,
one has
\begin{equation}
\theta(a_{1},a_{2},...,a_{n},\theta(b_{1},b_{2},...,b_{n},c))=
\end{equation}
$$=\theta(\theta(a_{1},a_{2},...,a_{n},b_{1}),\theta(a_{1},a_{2},...,a_{n},b_{2}),...,\theta(a_{1},a_{2},...,a_{n},b_{n}),c).$$
\end{lem}

       \begin{definition} Let $A$ be a set and let $\theta$ be an $(n+1)$-ary operation on $A$. $\theta$ is called 2-associative
        if the equivalent conditions of Lemma 3.1 are satisfied.\end{definition}
       When $n=1$, the 2-associativity obviously is equivalent to the associativity in the usual
       sense.
       \vskip+3mm

       \begin{rem}
       We left the term "1-associativity" for the straightforward
       generalization of the usual associativity condition given by
       the move of parentheses, i.e. for the following condition: moving the internal symbol $\theta$ together with the attached parentheses in
the expression
 $\theta(a_{1},a_{2},...,a_{n},\theta({b_{1},b_{2},...,b_{n},c}))$ to any place,
 one obtains the same element, for any  $a_{1},a_{2},...,a_{n},b_{1},b_{2},...,b_{n},c\in A$.

An example of an 1-associative $(n+1)$-ary operation is given by the
operation
$$\theta(a_{1},a_{2},...,a_{n},b)=a_{1}a_{2}
...a_{n}b$$

\noindent on a semigroup.\vskip+2mm
\end{rem}

       Let us now give examples of 2-associative operations.

       \begin{exmp} (a) Let $A$ be a set. The operation $\theta$ on the set $Map(A^{n},A)$
defined above is 2-associative. More generally, let $\mathcal{C}$ be
a category with finite products. Consider any object $C$ and
$Mor(C^{n},C)=\{f:C^{n}\rightarrow C\}$, with the operation $\theta$
defined similarly to the case of a set $C$. The operation $\theta$
is 2-associative.\vskip+2mm

(b) Let $A$ be any set, and $i$ and $n$ be any natural numbers with
$1\leq i\leq (n+1)$. Let
$$\theta(a_{1},a_{2},...,a_{n},a_{n+1})=a_{i}.$$

\noindent Then $\theta$ is a 2-associative operation.\vskip+2mm

(c) Let A be a semigroup, and $i$ and $n$ be any natural numbers
with $1\leq i\leq n$. Let us introduce the $(n+1)$-ary operation by
$$\theta(a_{1},a_{2},...,a_{n},b)= a_{i}b.$$
\noindent Then $\theta$ is 2-associative. \vskip+2mm

(d) Let $A$ be the set of all $(n+1)\times (n+1)$ matrices. Let
$\theta$ sends an $(n+1)$-tuple $(M_{1}, M_{2},...,M_{n+1})$ to the
matrix whose $i$th row is the $i$th row of the matrix $M_{i}$. This
operation is 2-associative.

\vskip+2mm

(e) Let $A$ be a commutative monoid such that the order of any
element divides $(n-1)$. Then the operation $\theta$ defined by

$$\theta(a_{1},a_{2},...,a_{n},b)=a_{1}+a_{2}+...+a_{n}+b$$

\noindent is both 1- and 2-associative.

\end{exmp}\vskip+3mm

          \begin{exmp} Let $A$ be a set, and $\theta$ be a ternary operation. Let, for any $x\in A$, $\circ_{x}$ be the binary operation
defined by
$$a\circ_{x} b=\theta(a,x,b).$$
\noindent Then $\theta$ is
2-associative if and only if
$$a\circ_{x}(b\circ_{y}c)=(a\circ_{x}b)\circ_{(a_{\circ_{x}}y)} c,$$

\noindent for any $a,b,c,x,y\in A$. We have at least two such
operations on a distributive lattice. One of them is defined by

\begin{equation}
\theta(a,b,c)=(a\vee b)\wedge c,
\end{equation}

\noindent and the other by

\begin{equation}
\theta(a,b,c)=(a\vee c)\wedge b.
\end{equation}

\noindent Recall that (3.3) is a promotodular operation in the
algebraic theory of the variety  of Boolean algebras
\cite{BC2}.\end{exmp}\vskip+3mm

To construct further examples of 2-assocative protomodular algebras,
below we give a simple lemma.

Let $\textbf{V}_{n}$ be the simplest protomodular  variety, i.e. the
variety with the signature  $\mathfrak{F}_{n}$ containing only one
$(n+1)$-ary operation symbol $\theta$, the binary operation symbols
$\alpha_{1},\alpha_{2},...,\alpha_{n}$, and the constants
$e_{1},e_{2},...,e_{n}$,
           where the identities are (2.1) and (2.2). Similarly, we denote by $\overline{\textbf{V}}_{n} $ the simplest
semi-abelian variety.

          \begin{lem} Let $A$ be a nonempty set, and let $\theta:A^{n+1}\rightarrow A$ be any mapping.

          (a) Let $e_{1},e_{2},...,e_{n}\in A$. There are binary operations $\alpha_{i}$ $(1\leq i\leq n)$ on $A$
such that

\begin{equation}
(A,\theta, (\alpha_{i})_{1\leq i\leq n}, (e_{i})_{1\leq i\leq
n})
\end{equation}
\vskip+2mm

\noindent  is a $\textbf{V}_{n}$-algebra if and only if, for any
$b\in A$, the mapping $\theta_{b}$ is surjective and

\begin{equation}
\theta(e_{1},e_{2},...,e_{n},b)=b.
\end{equation}
\vskip+2mm

In particular, if $\theta$ satisfies (3.5) and the identity

\begin{equation}
\theta(a,a,...,a,b)=a,
\end{equation}

\noindent then (3.4) is a $\textbf{V}_{n}$-algebra\footnote{If
$e_{1}=e_{2}=...=e_{n}$, then (3.5) and (3.6) imply that $A$ is one
element set.}, for some $\alpha_{i}$ $(1\leq i\leq n)$. \vskip+2mm

(b) Let $e\in A$. There are binary operations $\alpha_{i}$ $(1\leq
i\leq n)$ on $A$ such that

\begin{equation}
(A, \theta, (\alpha_{i})_{1\leq i\leq
n}, e)
\end{equation}
\vskip+2mm

\noindent is an $\overline{\textbf{V}}_{n} $-algebra if and only if
for any $b\in A$, the mapping $\theta_{b}$ is surjective and

\begin{equation}
\theta(e,e,...,e,b)=b.
\end{equation}
\vskip+2mm

\end{lem}
\vskip+3mm

 \begin{exmp}
  Lemma 3.6 implies that if A from (c) of Example 3.4 is a group, then  it can be turned into a 2-associative $\overline{\textbf{V}}_{n}
$-algebra, for any $n$. Now consider any groups $A_{1}, A_{2},...,
A_{k}$. $A_{1,i_{1}}\times A_{2,i_{2}}\times...\times A_{k,i_{k}}$
can be turned into a 2-associative $\overline{\textbf{V}}_{n}
$-algebra, for any naturals $i_{1},i_{2},..., i_{k}$ not exceeding
$n$. Here $A_{j,i_{j}}$ denotes the set $A_{j}$ with the operation
defined in (c) of Example 3.4 for $i=i_{j}$.\vskip+2mm

\end{exmp}

\vskip+3mm

\begin{exmp}
Any distributive lattice with top and bottom elements has the
structure of a 2-associative $\textbf{V}_{2}$-algebra. Indeed,
$\theta$ given by (3.3) satisfies (3.5) and (3.6).
\end{exmp}
\vskip+3mm

\begin{exmp}
Let $\mathcal{C}$ be a category with finite products and let
$Mor_{\Delta}(C^{n},C)$ denote the set of retractions of the
diagonal morphism $\triangle:C\rightarrow C^{n}$. It is closed under
the operation $\theta$ considered in Example 3.4(a). Let $e_{i}$ be
the $i$-th projection $\pi_{i}:C^{n}\rightarrow C$. Then (3.5) and
(3.6) are satisfied since
$$\theta(f,f,...,f,g)=g(f,f,...,f)=g\triangle f=f.$$

\noindent In this way $Mor_{\Delta}(C^{n},C)$ turns into a
2-associative $\textbf{V}_{n} $-algebra.
\end{exmp}
\vskip+2mm

\begin{rem}

 One can show that the semi-abelian operations of
Heyting algebras, Heyting semi-lattices, and locally Boolean
distributive algebras given in \cite{J}, \cite{BC1} are neither
1-associative nor 2-associative, provided that the algebras are not
trivial. Moreover, neither of the operations given by (3.2) and
(3.3) on a non-trivial lattice is 1-associative. However, as noted
in Example 3.5, both operations are 2-associative.

\end{rem}
\vskip+3mm

\begin{rem}
A Mal'cev operation can be 1-associative (for instance, in the
variety of groups). However, it is 2-associative only for trivial
algebras\footnote{Indeed, we have $\mu(a,b,\mu(b,b,b))=a$ and
$\mu(\mu(a,b,b),\mu(a,b,b),b)=b$.}.
\end{rem}

To give further negative examples, let us introduce the notion of a
strict protomodular algebra. First observe that the identity (2.2) implies that
the equation

\begin{equation}
\theta_{b}(x_{1},x_{2},...,x_{n})=\theta(x_{1},x_{2},...,x_{n},b)=a
\end{equation}

\noindent has a solution, for any $a$, $b\in A$. \vskip+3mm

\begin{lem} Let $A$ be an algebra from a protomodular variety. The following conditions are equivalent:

(i) $\theta_{b}$ is a bijection, for any $b\in A$;\vskip+2mm

(ii) the equation (3.9) has a unique solution for any $a,b\in
A$;\vskip+2mm

(iii) the system of equations

$$\alpha_{i}(x,b)=a_{i}$$

has a (unique) solution, for any $b, a_{i}\in A$ $(1\leq i\leq
n)$.\vskip+2mm

(iv) the following identity is satisfied in $A$:

\begin{equation}
\alpha_{i}(\theta(a_{1},a_{2},...,a_{n},b),b)=a_{i},
\end{equation}

for any $i$ (1$\leq i\leq n$).

\end{lem}\vskip+3mm

\begin{definition} A protomodular operation $\theta$ is called strict if the
equivalent conditions of Lemma 3.12 are satisfied. If it is clear
which protomodular operation we are dealing with, then the term
"strict" will be referred to an algebra.
\end{definition}\vskip+2mm

Obviously, if $n>2$, then any strict protomodular algebra is
infinite.\vskip+2mm

\begin{exmp} $\textbf{V}_{n}$-algebras with the strict operation $\theta$ from the signature can be easily described. Consider any infinite set $A$, and,
 for any $b\in A$,
choose a bijection $\theta_{b}:A\times A\times ...\times
A\rightarrow A$ such that

$$\theta_{b}(e_{1},e_{2},...,e_{n})=b.$$

\noindent  Let us take
$$\alpha_{i}(a,b)=\pi_{i}\theta^{-1}_{b}(a),$$
\noindent  where $\pi_{i}$ is the $i$-th projection $A\times A\times
...\times A\rightarrow A$. In this way $A$ turns into a strict
$\textbf{V}_{n}$-algebra. Obviously, any strict
$\textbf{V}_{n}$-algebra can be given in this way.
\end{exmp}\vskip+5mm

Let $\textbf{V}$ be a protomodular variety, and  $\theta$ be its
protomodular operation.
\begin{exmp} Let $n=1$. A protomodular operation $\theta$ on an algebra $A$ is strict if and only
if the quadriple $\mathbb{A}=(A,\theta,\alpha_{1},1)$ is a left
semiloop. This in particular implies that a 2-associative
protomodular operation $\theta$ is strict if and only if
$\mathbb{A}$ is a group.
\end{exmp}\vskip+2mm

\begin{rem}
Let $n\geq 2$. Any strict protomodular algebra which is either
1-associative or 2-associative is trivial. Indeed, let $A$ be a
strict 1-associative protomodular algebra. Consider any $a\in A$.
From (2.3), we obtain
\begin{equation}
\theta(a,a,...,a,a)=\theta(\theta(a,a,...,a,e_{1}),e_{2},...,e_{n},a).
\end{equation}

\noindent Since $\theta_{a}$ is bijective, $a=e_{2}$, and hence the
algebra $A$ is trivial.

Let now $A$ be a strict 2-associative algebra, and let $a_{1},a_{2},...,a_{n},b\in A$. From Lemma 3.18 below we obtain

\begin{equation}
\theta(a_{1},a_{2},...,a_{n},e_{i})=a_{i},
\end{equation}

\noindent for any $i$ $(1\leq i\leq n)$. Now take any $i$ and $j$
with $1\leq i<j\leq n$. (3.12) implies that
$$\theta(a_{1},a_{2},...,a_{i},...,a_{j},...,a_{n},e_{i})=\theta(a_{1},a_{2},...,a_{i},...,e_{j},...,a_{n},e_{i}).$$

\noindent Hence $a_{j}=e_{j}$, and
$A$ is trivial.
\end{rem}
\vskip+3mm

\begin{rem}
Remark 3.16 and Example 3.5 imply that the protomodular operation
(3.3) on a Boolean algebra  mentioned in Section 2 is strict if and
only if the algebra is trivial. One can show that, although the
semi-abelian operations on non-trivial Heyting algebras, Heyting
semi-lattices, and locally Boolean distributive lattices given in
\cite{J} and \cite{BC1} are not 2-associative, they are not strict.
\end{rem}
\vskip+2mm

From (2.3) we obtain

\begin{lem} Let $A$ be a 2-associative algebra. For any $a_{1},a_{2},...,a_{n}$ and $b\in A$, one has

$$\theta(a_{1},a_{2},...,a_{n},b)=\theta(\theta(a_{1},a_{2},...,a_{n},e_{1}),\theta(a_{1},a_{2},...,a_{n},e_{2}),...,$$
$$\theta(a_{1},a_{2},...,a_{n},e_{n}),b).$$

\end{lem}

\vskip+3mm

\section{Associative Semi-Abelian Algebras and Their Groups}

From the fact that any set equipped with an associative binary
operation, which has a left identity and left inverses, is a group
(see, for instance, \cite{F}), it immediately follows that an
associative semi-abelian operation with $n=1$ is a group operation.
In this section we study the question when an algebraic theory of a
protomodular variety contains a group operation in the case of an
arbitrary $n$.

         From now on, unless specified otherwise, we will deal with an arbitrary protomodular
         variety.
         \vskip+2mm
         \begin{lem} Let $A$ be a 2-associative protomodular algebra. Then, for any $b,c\in A$, there is $a\in A$ such that
         \begin{equation}
\theta(a,a,...,a,b)=c.
         \end{equation}
\noindent  The element $a$ can be taken as
\begin{equation}
a=\theta(\alpha_{1}(c,\theta(b,b,...,b)),\alpha_{2}(\theta(c,\theta(b,b,...,b))...,\alpha_{n}(c,
\theta(b,b,...,b)),b).
\end{equation}
         \end{lem}

         \begin{proof} The statement of this lemma immediately follows from the
identities (3.1) and (2.2). Below we give another proof of the
existence, it enables us to avoid cumbersome formulae.

Since the mapping $F:A\rightarrow Map(A^{n},A)$, $b\mapsto
\theta_{b}$, preserves $\theta$, $F(A)$ is closed in $Map(A^{n},A)$
under the operation introduced in Section 3.
         Therefore, for any $b\in A$
, the mapping
$f=\theta_{b}(\theta_{b},\theta_{b},...,\theta_{b}):A^{n}\rightarrow
A$ lies in $F(A)$, and hence is equal to $\theta_{d}$ for some $d\in
A$. This implies that $f$ is surjective. Then, for any $c\in A$,
there is $(a_{1},a_{2},...,a_{n})\in A^{n}$ such that
$$\theta_{b}(\theta_{b}(a_{1},a_{2},...,a_{n}),\theta_{b}(a_{1},a_{2},...,a_{n}),...,\theta_{b}(a_{1},a_{2},...,a_{n}))=c.$$

\noindent This implies (4.1) for

$$a=\theta(a_{1},a_{2},...,a_{n},b).$$





\end{proof}

      One can easily verify
      \begin{lem}Let $A$ be a 2-associative protomodular algebra. Then the binary operation
      \begin{equation}
      ab=\theta(a,a,...,a,b).
      \end{equation}

\noindent is associative. \end{lem}\vskip+3mm

    \begin{prop} Let $\theta$ be a 2-associative semi-abelian operation on an algebra $A$.
    Then  $A$   with the binary operation defined by (4.3) is a group. For any $b\in A$, we have

\begin{equation}
b^{-1}=
\end{equation}
$$=\theta(\alpha_{1}(e,\theta(b,b,...,b)),\alpha_{2}(\theta(e,\theta(b,b,...,b))...,\alpha_{n}(e,
\theta(b,b,...,b)),b).$$
    \end{prop}

     \begin{proof} The equality (2.3) implies that $e$ is a left unit.
     Now it suffices  to apply  Lemma 4.1, Lemma 4.2 and the above-mentioned statement from \cite{F}.
     \end{proof}\vskip+3mm

\begin{rem}
Note that, for any 2-associative protomodular operation $\theta$,
the operation given by (4.3) does not in general define the group
structure. For the counter-example take the protomodular operation
 given by (3.3) on a Boolean algebra.
\end{rem}

Proposition 4.3 immediately gives rise to

\begin{cor} In the conditions of Proposition 4.3, we
have:

(a) $\theta(a,a,...,a,e)=a$, for any  $a\in A$;

(b) for any $b,c\in A$, there is a unique $a\in A$ with
$\theta(a,a,...,a,b)=c$;\vskip+2mm

(c) for any $a,c\in A$, there is a unique $b\in A$ with
$\theta(a,a,...,a,b)=c.$
\end{cor}\vskip+2mm

\vskip+2mm
\begin{theo}
For a variety $\textbf{V}$ of universal algebras, the following
conditions are equivalent:\vskip+2mm

(i) An algebraic theory of $\textbf{V}$ contains a group
operation;\vskip+2mm

(ii) An algebraic theory of $\textbf{V}$ contains a constant and a
Mal'cev operation $\mu$ which is associative in the sense of
\cite{JP}, i.e. satisfies the following identity

$$\mu(a,b,\mu(c,d,x))=\mu(\mu(a,b,c),d,x);$$
\vskip+2mm

(iii) An algebraic theory of $\textbf{V}$ contains a semi-abelian
operation which is 2-associative; \vskip+2mm

(iv) An algebraic theory of $\textbf{V}$ contains a protomodular
operation $\theta$ which satisfies the following identity

$$\theta(\alpha_{1}(a_{1},a_{2}),...,\alpha_{n}(a_{1},a_{2}),\theta(\alpha_{1}(b_{1},b_{2}),...,\alpha_{n}(b_{1},b_{2}),c))=$$
$$=\theta(\alpha_{1}(\theta(\alpha_{1}(a_{1},a_{2}),...,\alpha_{n}(a_{1},a_{2}),b_{1}),b_{2}),$$
$$\alpha_{2}(\theta(\alpha_{1}(a_{1},a_{2}),...,\alpha_{n}(a_{1},a_{2}),b_{1}),b_{2}),$$
$$...$$
$$\alpha_{n}(\theta(\alpha_{1}(a_{1},a_{2}),...,\alpha_{n}(a_{1},a_{2}),b_{1}),b_{2}),c),$$
\vskip+2mm

\noindent for the corresponding binary operations $\alpha_{i}$.
\end{theo}

\begin{proof} The equivalence of (i) and (ii) was given in \cite{JP}; we already have mentioned it in the Introduction.
The equivalence of (i) and (iii) immediately
follows from Proposition 4.3.  For (i)$\Leftrightarrow$(iv) let us
observe that a protomodular variety has the Mal'cev term
$\mu(a,b,c)=\theta(\alpha_{1}(a,b),\alpha_{2}(a,b),...,\alpha_{n}(a,b),c)$.
\end{proof}

According to Proposition 4.3 we have the functor
$$R:2-Ass\textbf{V}\rightarrow \textbf{Grp}$$

\noindent where $\textbf{V}$ is a semi-abelian variety,
$2$-$Ass\textbf{V}$ denotes the category of 2-associative
$\textbf{V}$-algebras, while $\textbf{Grp}$ denotes the category of
groups; $R$ sends a 2-associative $\textbf{V}$-algebra to itself
with the group structure introduced above. The functor $R$ obviously
has a left adjoint.

\vskip+3mm
 At the end of the paper we give the description of 2-associative $\overline{\textbf{V}}_{n}$-algebras as groups
with some additional structure.
      Let $n$ be a natural number with $n\geq 2$. We define an $n$-enriched group as a triple $(G,\gamma, (\alpha_{i})_{1\leq i\leq n})$, where $G$ is a group,
      $\gamma$ is a mapping (not necessarily a homomorphism) $G^{n}\rightarrow G$, and $\alpha_{i}$ is a binary operation on $G$, such that
      $$\gamma(\alpha_{1}(a,b), \alpha_{2}(a,b),...,\alpha_{n}(a,b))b=a;$$
      $$\alpha_{i}(a,a)=e,$$

\noindent and  the following  distributivity condition is satisfied:
$$\gamma(a_{1},a_{2},...,a_{n})\gamma(b_{1},b_{2},...,b_{n})=$$
$$=\gamma(\gamma(a_{1},a_{2},...,a_{n})b_{1},
\gamma(a_{1},a_{2},...,a_{n})b_{2},...,\gamma(a_{1},a_{2},...,a_{n})b_{n}).$$\vskip+3mm

      Let $\textbf{EnrGrp}_{n}$ be the category,
      whose objects  are $n$-enriched groups and morphisms are group homomorphisms preserving $\gamma$ and all $\alpha_{i}$.

    Lemma 3.18 implies
    \begin{theo} The categories $2$-$Ass\overline{\textbf{V}}_{n}$ and $\textbf{EnrGrp}_{n}$ are isomorphic.
    \end{theo}
    \begin{proof} Let $F:2$-$Ass\overline{\textbf{V}}_{n}\rightarrow \textbf{EnrGrp}_{n}$
be the  functor sending an algebra $A$ to the set $A$ with the group
structure described above. Moreover, let

$$\gamma(a_{1},a_{2},...,a_{n})=\theta(a_{1},a_{2},...,a_{n},e).$$

\noindent Then, as it follows from Corollary 4.5(a), $(G,\gamma,
(\alpha_{i})_{1\leq i\leq n})$ is an enriched group. Consider the
functor
$$G:\textbf{EnrGrp}_{n}\rightarrow 2-Ass\overline{\textbf{V}}_{n},$$

\noindent sending $(G,\gamma, (\alpha_{i})_{1\leq i\leq n})$ to the
set $G$ equipped with the $(n+1)$-ary operation defined as

$$\theta(a_{1},a_{2},...,a_{n},b)=\gamma(a_{1},a_{2},...,a_{n})b.$$

\noindent Then $G$ turns into a 2-associative
$\overline{\textbf{V}}_{n}$-algebra. One can easily verify that $F$
and $G$ are the mutually inverse functors.
\end{proof}

\section{Acknowledgment}
This work is supported by Shota Rustaveli  National Science
Foundation (Ref.: FR-18-10849).

Authors address: \vskip+2mm

A. Razmadze Mathematical Institute of Tbilisi State University

6 Tamarashvili Str., Tbilisi, 0177, Georgia

E-mail: dalizan@rmi.ge


\begin{thebibliography}{99}






\bibitem{BC1} F. Bourceux, M. M. Clementino, Topological semi-Abelian algebras, {\em Adv. Math}, 190(2005), 425-453.

\bibitem{BC2} F. Bourceux, M. M. Clementino, Topological protomodular algebras, {\em Topology and its applications},
153(2006), 3085-3100.


\bibitem{B} D. Bourn, Normalization equivalence, kernel equivalence and affine categories, {\em Lecture notes in Math}, 1448(1991), 43-62.


\bibitem{BJ} D. Bourn, G. Janelidze, Characterization of protomodular varieties of universal algebras, {\em Theory Appl. Categ.}, 11(2002), 143-147.



\bibitem{F} J. B. Fraleigh, A first course in Abstract
Algebra, Person, 2004.


\bibitem{JMT} G. Janelidze, L. Marki, W. Tholen, Semi-abelian categories, {\em J. Pure Appl. Algebra}, 168(2002), 367-386.

\bibitem{JP} P. T. Johnstone and M. C. Pedicchio, Remarks on continous Mal'cev
algebras, {\em  Rend. Instit. Mat. Univ. Trieste} 25(1995), 277-287.


\bibitem{J} P. T. Jonstone, A note on the semiabelian variety of Heyting semilattices, in: Galois theory, Hopf algebras, and Semiabelian categories,
Fields Intst. Commun. 43(2004), 317-318.


\bibitem{U} A. Ursini, Osservazioni sulla varieta BIT, {\em Boll. Un. Mat.
Ital.} (4)7(1973), 205-211.









\end{thebibliography}
\end{document}